\documentclass[11pt]{article}

\usepackage{amsmath}
\usepackage{amsfonts}
\usepackage{color,bm}
\usepackage{amssymb}
\usepackage{graphicx}
\usepackage{hyperref}
\usepackage{enumerate}
\usepackage[all]{xy}

\def\bA{{\bm{A}}}
\def\bB{{\bm{B}}}

 \allowdisplaybreaks

\begin{document}

\newtheorem{problem}{Problem}

\newtheorem{theorem}{Theorem}[section]
\newtheorem{corollary}[theorem]{Corollary}
\newtheorem{definition}[theorem]{Definition}
\newtheorem{conjecture}[theorem]{Conjecture}
\newtheorem{question}[theorem]{Question}
\newtheorem{lemma}[theorem]{Lemma}
\newtheorem{proposition}[theorem]{Proposition}
\newtheorem{quest}[theorem]{Question}
\newtheorem{example}[theorem]{Example}

\newenvironment{proof}{\noindent {\bf
Proof.}}{\rule{2mm}{2mm}\par\medskip}

\newenvironment{proofof3}{\noindent {\bf
Proof of  Theorem 1.2.}}{\rule{2mm}{2mm}\par\medskip}

\newenvironment{proofof5}{\noindent {\bf
Proof of  Theorem 1.3.}}{\rule{2mm}{2mm}\par\medskip}

\newcommand{\remark}{\medskip\par\noindent {\bf Remark.~~}}
\newcommand{\pp}{{\it p.}}
\newcommand{\de}{\em}

\title{  {An Oppenheim type determinantal inequality 
for the Khatri-Rao product}
\thanks{This work was supported by  NSFC (Grant Nos. 11671402, 11871479).
\newline 
 E-mail addresses: ytli0921@hnu.edu.cn (Y. Li), 
fenglh@163.com (L. Feng, corresponding author).} }

\author{Yongtao Li$^{a}$,  Lihua Feng$^{\dag,b}$\\
{\small ${}^a$School of Mathematics, Hunan University} \\
{\small Changsha, Hunan, 410082, P.R. China } \\
{\small $^b$School of Mathematics and Statistics, Central South University} \\
{\small New Campus, Changsha, Hunan, 410083, P.R. China. } }

\maketitle

\vspace{-0.5cm}

\begin{abstract}
We first give an Oppenheim type determinantal inequality 
for the Khatri-Rao product of two block positive semidefinite matrices, 
and then we extend  our result to multiple block matrices. 
As products, the extensions of  Oppenheim type inequalities for 
the Hadamard product are also included. 
 \end{abstract}

{2010 Mathematics Subject Classication.  15A45, 15A60, 47B65.}

\section{Introduction}

\label{sec1} 

We use the following standard notation. 
The set of $m\times n$ complex matrices is denoted by $\mathbb{M}_{m\times n}(\mathbb{C})$, 
or simply by $\mathbb{M}_{m\times n}$,  
when $m=n$, we put $\mathbb{M}_n$ for $\mathbb{M}_{n\times n}$. 
The identity matrix of order $n$ by  $I_n$, or $I$ for short. 
If $A=[a_{ij}]$ is of order $p\times q$ and  
 $B$ is of order $r\times s$, the Kronecker product (tensor product) of $A$ with $B$, 
denoted by $A\otimes B$, is an $pr\times qs$ matrix, 
partitioned into $p\times q$ block matrix 
with the $(i,j)$-block the $r\times s$ matrix $a_{ij}B$, 
i.e.,  $A\otimes B=[ a_{ij}B ]_{i,j=1}^{p,q}$. 
Given two matrices $A=[a_{ij}]$ and $B=[b_{ij}]$ with the same order, 
the Hadamard product of $A,B$ is defined as $A\circ B=[a_{ij}b_{ij}]$. 
It is easy to see that $A\circ B$ is a principal submatrix of $A\otimes B$. 
By convention, the $\mu \times \mu$ leading principal submatrix of $A$ 
is denoted by $A_{\mu}$.

Let  $A=[a_{ij}]\in \mathbb{M}_n$ be  positive semidefinite. 
 The  Hadamard inequality says that 
\begin{equation}
 \prod_{i=1}^n a_{ii} \ge \det A. 
\end{equation} \label{eqhada}
If $B=[b_{ij}]\in \mathbb{M}_n$ is positive semidefinite, 
it is well-known that $A\circ B$ is positive semidefinite. Moreover, 
the celebrated Oppenheim inequality (see \cite{Oppenheim30} or \cite[p. 509]{HJ13}) states that 
 \begin{equation} \label{eq1}
\det (A\circ B) \ge \det A\cdot \prod_{i=1}^n b_{ii} \ge \det (AB). 
\end{equation}
Setting $B=I_n$, then (\ref{eq1}) reduces to (\ref{eqhada}). 
Note that $A\circ B=B\circ A$, thus we also have 
\begin{equation} \label{eq2}
 \det (A\circ B) \ge \det B\cdot \prod_{i=1}^n a_{ii} \ge \det (AB).  
\end{equation}
The following  inequality (\ref{eq3}) not only generalized Oppenheim's result, 
but also presented a well connection between  (\ref{eq1}) and (\ref{eq2}); 
see \cite[Theorem 3.7]{Styan73} for more details.  
\begin{equation} \label{eq3}
\det (A\circ B) + \det (AB) \ge \det A\cdot \prod_{i=1}^n b_{ii} + \det B\cdot \prod_{i=1}^n a_{ii}. 
\end{equation}
Inequality (\ref{eq3}) is usually called Oppenheim-Schur's inequality. 
Furthermore, Chen \cite{Chen03}  generalized  (\ref{eq3}) and proved an implicit improvement, 
i.e.,  if $A$ and $B$ are $n\times n$ positive definite matrices, then 
\begin{equation} \label{eq4}
\det (A\circ B) \ge \det (AB)\prod_{\mu=2}^n 
\left( \frac{a_{\mu\mu}\det A_{\mu-1}}{\det A_{\mu}} + 
\frac{b_{\mu\mu}\det B_{\mu-1}}{\det B_{\mu}} -1  \right), 
\end{equation}
where $A_\mu$ and $B_{\mu}$ denote the $\mu \times \mu$ leading principal submatrices 
of $A$ and $B$, respectively.

Over the past years, various generalizations and extensions of 
(\ref{eq3}) and (\ref{eq4}) have been obtained in 
the literature. 
For instance, see \cite{Zhang04, Zhang09} for the equality cases; 
see \cite{Ando80, LZ97, YL00, Chen07} for the extensions of $M$-matrices; 
see \cite{GK12,Lin14,DL20}  for the extensions of block Hadamard product.

\vspace{0.3cm}

In this paper, we are mainly concentrated on  block positive semidefinite matrices. 
Let $\mathbb{M}_n(\mathbb{M}_{p\times q})$ be the set of complex matrices 
partitioned into $n\times n$ blocks with each block being a $p\times q$ matrix. 
The element of $\mathbb{M}_n(\mathbb{M}_{p\times q})$ is usually 
written as the bold letter $\bA=[A_{ij}]_{i,j=1}^n$, 
where $A_{ij}\in \mathbb{M}_{p\times q}$ for all $1\le i,j\le n$. 
For $\bm{A}=[A_{ij}]\in \mathbb{M}_n(\mathbb{M}_{p\times q}) $ and $\bB=[B_{ij}]\in \mathbb{M}_n(\mathbb{M}_{r\times s})$, 
the Khatri-Rao product $\bA *\bB$, first introduced  in \cite{KR68}, 
is given by $\bA *\bB:=[A_{ij}\otimes B_{ij}]_{i,j=1}^n$, 
where $A_{ij}\otimes B_{ij}$ denotes the Kronecker product of $A_{ij}$ and $B_{ij}$. 
Clearly, when $p=q=r=s=1$, that is, $\bA$ and $\bB$ are $n\times n$ matrices with complex entries, 
the Khatri-Rao product coincides with the classical Hadamard product; 
when $n=1$, it is identical with the usual Kronecker product. 
It is easy to verify that 
$(\bA * \bB) *\bm{C}=\bA * (\bB *\bm{C})$, 
so the Khatri-Rao product of $\bA^{(1)},\ldots ,\bA^{(m)}$ could be denoted by 
$\prod_{i=1}^m*\bA^{(i)}$. 
We refer to \cite{Liu99, Liu02, Liu08}  for more properties of Khatri-Rao product.

Recently, Kim et al. \cite{Kim17} gave the following extension of 
Chen's result (\ref{eq4}) for the Khatri-Rao product, 
if $\bA,\bB\in \mathbb{M}_n(\mathbb{M}_k)$ are positive definite, then 
\begin{equation}  \label{eqkim}
\begin{aligned}
 \det (\bm{A}* \bm{B}) & \ge 
(\det \bm{A} \bm{B})^k \\
&\quad \times  \prod_{\mu=2}^n 
\Bigg( \Bigl( \frac{\det A_{\mu \mu} \det \bm{A}_{\mu-1}}{\det \bm{A}_{\mu}}\Bigr)^k 
 + \frac{\det B_{\mu \mu} \det \bm{B}_{\mu-1}}{\det \bm{B}_{\mu}}\Bigr)^k -1 \Bigg), 
\end{aligned}
\end{equation}
where $\bA_{\mu}=[A_{ij}]_{i,j=1}^{\mu}$ and $\bB_{\mu}=[B_{ij}]_{i,j=1}^{\mu}$ stand 
for the $\mu \times \mu$ leading principal block submatrices of $\bA$ and $\bB$, respectively. 

The paper is organized as follows. 
We first modify Kim's result (\ref{eqkim}) to more general setting 
where the blocks in each $n\times n$ block  matrix are of different order.  
Motivated by the works in \cite{FL16} and \cite{DL20}, 
we then show extensions of our results to multiple block positive semidefinite matrices. 
Our results extend the above mentioned results (\ref{eq3}), (\ref{eq4}) 
and (\ref{eqkim}).

\section{Main result}
\label{sec2}

To review the proof of (\ref{eqkim}) in \cite{Kim17}, 
we present a slightly more general result (Theorem \ref{thm21}). 
Clearly, when $p=q=k$, Theorem \ref{thm21} reduces to (\ref{eqkim}). 
Such a generalization also actuates our cerebration and 
propels the extension (Theorem \ref{thm24}). 
Because the lines of proof between Theorem \ref{thm21} and (\ref{eqkim})  are  similar, 
so we leave the details for the interested readers.

\begin{theorem} \label{thm21}
Let $\bm{A}=[A_{ij}]\in \mathbb{M}_n(\mathbb{M}_p),\bm{B}=[B_{ij}]\in  \mathbb{M}_n(\mathbb{M}_q)$ 
be positive definite. Then 
\begin{equation*}
\begin{aligned}
 \det (\bm{A}* \bm{B}) & \ge 
(\det \bm{A})^q(\det \bm{B})^p \\
&\quad \times  \prod_{\mu=2}^n 
\Bigg( \Bigl( \frac{\det A_{\mu \mu} \det \bm{A}_{\mu-1}}{\det \bm{A}_{\mu}}\Bigr)^q 
 + \frac{\det B_{\mu \mu} \det \bm{B}_{\mu-1}}{\det \bm{B}_{\mu}}\Bigr)^p-1 \Bigg), 
\end{aligned}
\end{equation*}
where $\bA_{\mu}=[A_{ij}]_{i,j=1}^{\mu}$ and $\bB_{\mu}=[B_{ij}]_{i,j=1}^{\mu}$ 
for every $\mu=1,2,\ldots ,n$.
\end{theorem}

The following Lemma \ref{lemfis} is called  Fischer's inequality, 
which is an improvement of Hadamard's inequality (\ref{eqhada}) 
for block positive semidefinite  matrices. 

\begin{lemma} (see \cite[p. 506]{HJ13} or \cite[p. 217]{Zhang11}) \label{lemfis}
If $A=\begin{bmatrix} A_{11} & A_{12} \\ A_{21} & A_{22} \end{bmatrix}$ is an 
$n\times n$ positive semidefinite 
matrix with diagonal blocks being square, then 
\[ \prod_{i=1}^n a_{ii}\ge  \det A_{11} \det A_{22} \ge \det A.  \]
\end{lemma}

Next, we need to introduce a numerical inequality, 
which could be found in \cite{DL20}. 
For completeness, we here include a proof for the convenience of readers. 

\begin{lemma} \label{lem}
If $\bigl( a^{(i)}_{1}, a^{(i)}_2,\ldots ,a^{(i)}_n\bigr)\in \mathbb{R}^n,i=1,\ldots ,m$ and  
$a_{\mu}^{(i)}\ge 1$ for all $i,\mu$, then 
\begin{equation*} 
 \prod_{\mu=1}^n \left( \sum_{i=1}^m a^{(i)}_{\mu} -(m-1)\right) 
\ge \sum_{i=1}^m \prod_{\mu=1}^n a^{(i)}_{\mu} -(m-1). 
\end{equation*}
\end{lemma}

\begin{proof}
We use induction on $n$. 
When $n=1$, there is nothing to show. 
Suppose that the required inequality is true for $n=k$. Then we consider the case $n=k+1$, 
\begin{align*}
& \prod_{\mu=1}^{k+1} \left( \sum_{i=1}^m a^{(i)}_{\mu} -(m-1)\right)  \\
&= \left( \sum_{i=1}^m a^{(i)}_{k+1} -(m-1)\right) \cdot 
\prod_{\mu=1}^k \left( \sum_{i=1}^m a^{(i)}_{\mu} -(m-1)\right) \\
&\ge \left( \sum_{i=1}^m a^{(i)}_{k+1} -(m-1)\right) \cdot 
\left( \sum_{i=1}^m \prod_{\mu=1}^k a^{(i)}_{\mu} -(m-1)\right) \\
&= \sum_{i=1}^m \prod_{\mu=1}^{k+1}a_{\mu}^{(i)} -(m-1) 
 + \sum_{i=1}^m \left(a_{k+1}^{(i)}-1\right) \!\biggl( \sum_{j=1,j\neq i}^m \prod_{\mu=1}^k 
a_{\mu}^{(j)}-(m-1)\biggr) \\
&\ge \sum_{i=1}^m \prod_{\mu=1}^{k+1}a_{\mu}^{(i)} -(m-1). 
\end{align*}
Thus, the required  holds for $n=k+1$, so the proof of  induction step is complete. 
\end{proof}

\noindent 
{\bf Remark.}~When $m=2$, Lemma \ref{lem}  implies that for every $a_{\mu},b_{\mu}\ge 1$, then 
\begin{equation} \label{eqlin}
  \prod_{\mu=1}^n (a_{\mu}+b_{\mu}-1) \ge \prod_{\mu=1}^n a_{\mu} + 
\prod_{\mu=1}^n b_{\mu} -1. 
\end{equation}
This inequality (\ref{eqlin}) plays an important role  in \cite{Lin14} 
for deriving determinantal inequalities, 
and we can see from  (\ref{eqlin}) that Chen's result (\ref{eq4}) 
is indeed an improvement of (\ref{eq3}). 
The stated proof of Lemma \ref{lem} is by induction on $n$. 
In fact, combining the above (\ref{eqlin}) and by induction on $m$, 
one could get another  way to prove  Lemma \ref{lem}.

The following Corollary \ref{coro} is a direct consequence from Lemma \ref{lem}, 
it will be used to facilitate the proof of Theorem \ref{thm24}. 

\begin{corollary} \label{coro}
If $b_1,b_2,\ldots ,b_m\in \mathbb{R}$ and $b_i\ge 1$ for all $i$, 
then for positive integer $q$
\[  \left(\sum_{i=1}^m b_i -(m-1)\right)^q \ge 
\sum_{i=1}^m b_i^q -(m-1). \] 
\end{corollary}

We give the following extension of Theorem \ref{thm21}. 

\begin{theorem} \label{thm24}
Let $\bA^{(i)}\in \mathbb{M}_n(\mathbb{M}_{q_i}),i=1,2,\ldots ,m$ be positive definite. Then 
\begin{equation*} 
\begin{aligned}
\det \left( \prod_{i=1}^m *\bA^{(i)}\right) 
&\ge \prod_{i=1}^m \left(\det \bA^{(i)}\right)^{\frac{q_1q_2\cdots q_m}{q_i}} \\
& \quad \times \prod_{\mu=2}^n  \left( \sum_{i=1}^m 
\left( \frac{\det A^{(i)}_{\mu \mu} \det \bm{A}^{(i)}_{\mu-1}}{\det \bm{A}^{(i)}_{\mu}} 
\right)^{\frac{q_1q_2\cdots q_m}{q_i}} - (m-1) \right), 
\end{aligned}
\end{equation*}
where $\bA^{(i)}_{\mu}$ stands for 
the $\mu \times \mu$ leading principal block submatrix of $\bA^{(i)}$. 
\end{theorem}

\begin{proof}
We show the proof by induction on $m$. 
When $m=2$, 
the required result degrades into Theorem \ref{thm21}. 
Assume that the required is true for the case $m-1$, that is 
\begin{equation*} 
\begin{aligned}
\det \left( \prod_{i=1}^{m-1} *\bA^{(i)}\right) 
&\ge \prod_{i=1}^{m-1} \left(\det \bA^{(i)}\right)^{\frac{q_1q_2\cdots q_{m-1}}{q_i}} \\
& \quad \times \prod_{\mu=2}^n  \left( \sum_{i=1}^{m-1} 
\left( \frac{\det A^{(i)}_{\mu \mu} \det \bm{A}^{(i)}_{\mu-1}}{\det \bm{A}^{(i)}_{\mu}} 
\right)^{\frac{q_1q_2\cdots q_{m-1}}{q_i}} - (m-2) \right). 
\end{aligned}
\end{equation*}
Now we consider the case $m>2$,  we have 
\begin{equation*} 
\begin{aligned}
&\det \left( \prod_{i=1}^m *\bA^{(i)}\right) \\
&=  \det \left( \Bigl(\prod_{i=1}^{m-1} *\bA^{(i)}\Bigr) * A^{(m)}\right)  \\
&\ge \left( \det \Bigl(\prod_{i=1}^{m-1} *\bA^{(i)}\Bigr)\right)^{q_m} 
\left( \det A^{(m)}\right)^{q_1q_2\cdots q_{m-1}} \\ 
&\quad \times 
\prod_{\mu=2}^n 
\left( \!\!\Biggl(\frac{\det \left(\prod\limits_{i=1}^{m-1} \!\!*\bA^{(i)}\right)_{\!\!\mu\mu} 
\!\! \!\det  \Bigl(\prod\limits_{i=1}^{m-1} \!\!*\bA^{(i)}\Bigr)_{\!\!\mu \!-\!1}}{
\det \Bigl(\prod\limits_{i=1}^{m-1} \!\!*\bA^{(i)}\Bigr)_{\!\!\mu}}\Biggr)^{\!\!q_m} \!\! + 
\left( \frac{\det A^{(m)}_{\mu \mu} \det \bA^{(m)}_{\mu -1}}{ 
\det \bA^{(m)}_{\mu}}\right)^{\!\!q_1q_2\cdots q_{m-1}}\!\! -1 \right) \\
&\ge \prod_{i=1}^m \left( \det \bA^{(i)}\right)^{\frac{q_1q_2\cdots q_m}{q_i}} 
 \times \prod_{\mu=2}^n  \left( \sum_{i=1}^{m-1} 
\left( \frac{\det A^{(i)}_{\mu \mu} \det \bm{A}^{(i)}_{\mu-1}}{\det \bm{A}^{(i)}_{\mu}} 
\right)^{\frac{q_1q_2\cdots q_{m-1}}{q_i}} - (m-2) \right)^{q_m} \\ 
&\quad \times 
\prod_{\mu=2}^n 
\left( \!\!\Biggl(\frac{\det \left(\prod\limits_{i=1}^{m-1} \!\!*\bA^{(i)}\right)_{\!\!\mu\mu} 
\!\! \!\det  \Bigl(\prod\limits_{i=1}^{m-1} \!\!*\bA^{(i)}\Bigr)_{\!\!\mu \!-\!1}}{
\det \Bigl(\prod\limits_{i=1}^{m-1} \!\!*\bA^{(i)}\Bigr)_{\!\!\mu}}\Biggr)^{\!\!q_m} \!\! + 
\left( \frac{\det A^{(m)}_{\mu \mu} \det \bA^{(m)}_{\mu -1}}{ 
\det \bA^{(m)}_{\mu}}\right)^{\!\!q_1q_2\cdots q_{m-1}}\!\! -1 \right).
\end{aligned}
\end{equation*}
For notational convenience, we denote  
\[ R_{\mu}:=\left( \sum_{i=1}^{m-1} 
\left( \frac{\det A^{(i)}_{\mu \mu} \det \bm{A}^{(i)}_{\mu-1}}{\det \bm{A}^{(i)}_{\mu}} 
\right)^{\frac{q_1q_2\cdots q_{m-1}}{q_i}} - (m-2) \right)^{q_m}, \]
and 
\[ S_{\mu}:= \left( \frac{\det \left(\prod\limits_{i=1}^{m-1} \!\!*\bA^{(i)}\right)_{\!\!\mu\mu} 
\!\! \det  \Bigl(\prod\limits_{i=1}^{m-1} \!\!*\bA^{(i)}\Bigr)_{\!\!\mu -1}}{
\det \Bigl(\prod\limits_{i=1}^{m-1} \!\!*\bA^{(i)}\Bigr)_{\!\!\mu}}\right)^{\!\!q_m} \!\! + 
\left( \frac{\det A^{(m)}_{\mu \mu} \det \bA^{(m)}_{\mu -1}}{ 
\det \bA^{(m)}_{\mu}}\right)^{\!\!q_1q_2\cdots q_{m-1}} \!\! -1. \]
By Fischer's inequality (Lemma \ref{lemfis}), we can see that 
\[ \det A^{(i)}_{\mu \mu} \det \bm{A}^{(i)}_{\mu-1} \ge \det \bm{A}^{(i)}_{\mu},
\quad i=1,2,\ldots ,m,  \]
which together with  Corollary \ref{coro} yields the following 
\begin{equation} \label{r1}
R_{\mu}  \ge  \sum_{i=1}^{m-1} 
\left( \frac{\det A^{(i)}_{\mu \mu} \det \bm{A}^{(i)}_{\mu-1}}{\det \bm{A}^{(i)}_{\mu}} 
\right)^{\frac{q_1q_2\cdots q_{m-1}q_m}{q_i}} - (m-2) \ge 1.
\end{equation}
On the other hand, by Fischer's inequality (Lemma \ref{lemfis}) again, we have 
\begin{gather*}
 \det \Bigl(\prod\limits_{i=1}^{m-1} \!\!*\bA^{(i)}\Bigr)_{\!\!\mu\mu} 
  \det  \Bigl(\prod\limits_{i=1}^{m-1} \!\!*\bA^{(i)}\Bigr)_{\!\!\mu -1}
 \ge \det \Bigl(\prod\limits_{i=1}^{m-1} \!\!*\bA^{(i)}\Bigr)_{\!\!\mu}.
\end{gather*}
Therefore, we obtain 
\begin{equation} \label{s1}
 S_{\mu} \ge \left( \frac{\det A^{(m)}_{\mu \mu} \det \bA^{(m)}_{\mu -1}}{ 
\det \bA^{(m)}_{\mu}}\right)^{\!\!q_1q_2\cdots q_{m-1}} \ge 1. 
\end{equation}
Since $R_{\mu}\ge 1$ and $S_{\mu}\ge 1$, which leads to 
\[ R_{\mu}S_{\mu} \ge R_{\mu} +S_{\mu} -1.\]
Hence, we get from (\ref{r1}) and (\ref{s1}) that 
\begin{align*}
\det \left( \prod_{i=1}^m *\bA^{(i)}\right)  
&\ge \prod_{i=1}^m \left( \det \bA^{(i)}\right)^{\frac{q_1q_2\cdots q_m}{q_i}}  
\prod_{\mu=2}^n R_{\mu} \prod_{\mu=2}^n S_{\mu} \\
&\ge \prod_{i=1}^m \left( \det \bA^{(i)}\right)^{\frac{q_1q_2\cdots q_m}{q_i}}  
\prod_{\mu=2}^n (R_{\mu} +S_{\mu} -1) \\
&\ge \prod_{i=1}^m \left(\det \bA^{(i)}\right)^{\frac{q_1q_2\cdots q_m}{q_i}} \\
& \quad \times \prod_{\mu=2}^n  \left( \sum_{i=1}^m 
\left( \frac{\det A^{(i)}_{\mu \mu} \det \bm{A}^{(i)}_{\mu-1}}{\det \bm{A}^{(i)}_{\mu}} 
\right)^{\frac{q_1q_2\cdots q_m}{q_i}} - (m-1) \right).
\end{align*}
This completes the proof. 
\end{proof}

Next, we will present the extension of 
Oppenheim type determinantal inequality (\ref{eq3}). 

\begin{theorem} \label{thm25}
Let $\bA^{(i)}\in \mathbb{M}_n(\mathbb{M}_{q_i}),i=1,2,\ldots ,m$ be positive demidefinite. Then 
\begin{equation} \label{eq12}
\begin{aligned}
&\det \left( \prod_{i=1}^m *\bA^{(i)}\right)  + (m-1) \prod_{i=1}^m \left(\det \bA^{(i)}\right)^{\frac{q_1q_2\cdots q_m}{q_i}} \\
&\quad \ge  \sum_{i=1}^m \prod_{j=1, j\neq i}^m
\Bigg( \det \bA^{(j)} \cdot \prod_{\mu=1}^n \det A^{(i)}_{\mu\mu} \Bigg)^{\frac{q_1q_2\cdots q_m}{q_i}} . 
\end{aligned}
\end{equation}
\end{theorem}

\begin{proof}
If any of $A^{(i)}_{\mu\mu}$ in (\ref{eq12}) is singular, then so is $\bA^{(i)}$. 
In this case, the right hand side of (\ref{eq12}) equal to zero. 
Indeed, by a standard perturbation argument, 
we may assume without loss of generality that all $\bA^{(i)}$ are positive definite. 
Thus, we may rewrite (\ref{eq12}) as the following 
\begin{equation} \label{eq13}
\det \left( \prod_{i=1}^m *\bA^{(i)}\right) 
\ge \prod_{i=1}^m \left(\det \bA^{(i)}\right)^{\!\!\frac{q_1q_2\cdots q_m}{q_i}} 
  \left( \sum_{i=1}^m 
\left( \frac{\prod_{\mu=1}^n \det A^{(i)}_{\mu \mu}}{\det \bm{A}^{(i)}} 
\right)^{\!\!\!\frac{q_1q_2\cdots q_m}{q_i}} \!\!- (m-1) \right). 
\end{equation}
By Fischer's inequality (Lemma \ref{lemfis}), we have 
\[ \det A^{(i)}_{\mu \mu} \det \bm{A}^{(i)}_{\mu-1}\ge \det \bm{A}^{(i)}_{\mu}.  \]
Therefore, it follows from Theorem \ref{thm24} and Lemma \ref{lem} that 
\begin{equation*} 
\begin{aligned}
& \det \left( \prod_{i=1}^m *\bA^{(i)}\right) \\
&\ge \prod_{i=1}^m \left(\det \bA^{(i)}\right)^{\frac{q_1q_2\cdots q_m}{q_i}}  
  \prod_{\mu=2}^n  \left( \sum_{i=1}^m 
\left( \frac{\det A^{(i)}_{\mu \mu} \det \bm{A}^{(i)}_{\mu-1}}{\det \bm{A}^{(i)}_{\mu}} 
\right)^{\!\!\frac{q_1q_2\cdots q_m}{q_i}} - (m-1) \right) \\
&\ge  \prod_{i=1}^m \left(\det \bA^{(i)}\right)^{\frac{q_1q_2\cdots q_m}{q_i}} 
\left( \sum_{i=1}^m  \prod_{\mu=2}^n  
\left( \frac{\det A^{(i)}_{\mu \mu} \det \bm{A}^{(i)}_{\mu-1}}{\det \bm{A}^{(i)}_{\mu}} 
\right)^{\!\!\frac{q_1q_2\cdots q_m}{q_i}} - (m-1) \right). 
\end{aligned}
\end{equation*}
Observe that 
\[ \prod_{\mu=2}^n \frac{\det A^{(i)}_{\mu \mu} \det \bm{A}^{(i)}_{\mu-1}}{\det \bm{A}^{(i)}_{\mu}}  
= \frac{\prod_{\mu=1}^n \det A^{(i)}_{\mu \mu}}{\det \bm{A}^{(i)}} .  \]
Hence, the proof of (\ref{eq13}) is complete. 
\end{proof}

In the sequel, by setting $q_1=q_2=\cdots =q_m=1$ in Theorem \ref{thm24} and Theorem \ref{thm25},  
we can get the following Corollary \ref{coro26} and Corollary \ref{coro27}
 for the Hadamard product, respectively. 
These two corollaries are extensions of Oppenheim-Schur's inequality (\ref{eq3}) 
and Chen's result (\ref{eq4}). 
The first  corollary can be found in  \cite[Theorem 7]{FL16} 
and the second one can be seen in \cite[Theorem 4]{DL20}.

\begin{corollary} \label{coro26}
Let $A^{(i)}\in \mathbb{M}_n(\mathbb{C}),i=1,2,\ldots ,m$ be positive definite. Then 
\begin{equation*}
\begin{aligned}
 \det \left( \prod_{i=1}^m \circ A^{(i)}\right) 
\ge  \left(\prod_{i=1}^m \det A^{(i)}\right)  \prod_{\mu=2}^n  \left( \sum_{i=1}^m 
 \frac{ a^{(i)}_{\mu \mu} \det {A}^{(i)}_{\mu-1}}{\det {A}^{(i)}_{\mu}}  - (m-1) \right), 
\end{aligned}
\end{equation*}
where $A^{(i)}_{\mu}$ stands for 
the $\mu \times \mu$ leading principal  submatrix of $A^{(i)}$. 
\end{corollary}

\begin{corollary} \label{coro27}
Let $A^{(i)}\in \mathbb{M}_n(\mathbb{C}),i=1,2,\ldots ,m$ be positive semidefinite. Then 
\begin{equation*}
\begin{aligned}
\det \left( \prod_{i=1}^m \circ A^{(i)}\right)  + (m-1) \prod_{i=1}^m \det A^{(i)}
 \ge  \sum_{i=1}^m \prod_{j=1, j\neq i}^m
 \det A^{(j)}\prod_{\mu=1}^n  a^{(i)}_{\mu\mu} . 
\end{aligned}
\end{equation*}
\end{corollary}

\section*{Acknowledgments}
This work was supported by  NSFC (Grant Nos. 11671402, 11871479).

\end{document}